\documentclass[12pt]{amsart}
\usepackage[colorlinks=true,
pdfstartview=FitV, linkcolor=blue, citecolor=red,
urlcolor=blue]{hyperref}
\usepackage{amsmath,amsfonts,latexsym,amssymb,mathrsfs}
\usepackage[latin1]{inputenc}
\newtheorem{theorem}{Theorem}
\newtheorem{lemma}[theorem]{Lemma}
\newtheorem{proposition}[theorem]{Proposition}

\newcommand{\seq}[1]{
\{{#1}_m\}_{m\in\mathbb \hN}}

\newcommand{\T}{\mathsf{T}}
\newcommand{\TM}{{\T}M}
\newcommand{\UpM}{\T_{+1}(M)}

\newcommand{\SOto}{\mathsf{SO}(2,1)}
\newcommand{\SOno}{\mathsf{O}(n-1,1)}

\newcommand{\R}{\mathbb{R}}
\newcommand{\Ht}{\mathsf{H}^2}
\newcommand{\tS}{\tilde{\Sigma}}
\newcommand{\U}{\mathsf{U}}
\newcommand{\Ss}{\mathsf{S}}
\newcommand{\US}{\U\Sigma}
\newcommand{\vv}{\mathsf{v}}
\newcommand{\dS}{\Ss_{+1}}
\newcommand{\Isom}{\mathsf{Isom}}

\newcommand\Ur{\U_{\mathsf{rec}}\Sigma}
\newcommand\tUr{\widetilde{\Ur}}

\newcommand\V{\V^{-}_{\mathsf{rec}}\Sigma}

\renewcommand\d{{\rm d}}
\renewcommand{\L}{\mathsf{L}}
\newcommand{\dG}{\L(\Gamma)}
\newcommand{\hN}{\widehat{N}}
\newcommand{\thN}{\widetilde{N}}
\newcommand{\EE}{\mathbb{E}}

\newcommand{\E}{\mathsf{E}}
\renewcommand{\V}{\mathsf{V}}

\newcommand{\dev}{\mathsf{dev}}
\newcommand{\Aff}{\mathsf{Aff}}

\title[Margulis Spacetimes]
{Geodesics in Margulis Spacetimes}
\author[Goldman]{William M. Goldman}
    \address{Department of Mathematics\\ University of Maryland\\
    College Park, MD 20742 USA}
    \email{wmg@math.umd.edu}
\author[Labourie]{
François LABOURIE }
    \address{
Univ.\  Paris-Sud, Laboratoire de Mathématiques, 
Orsay F-91405 Cedex; CNRS, Orsay cedex, F-91405}
\email{francois.labourie@math.u-psud.fr}

\date{\today}
\begin{document}

\maketitle
\begin{center} {\em Dedicated to the memory of Dan Rudolph}
\end{center}
\begin{abstract}
Let $M^3$ be a Margulis spacetime whose associated
complete hyperbolic surface $\Sigma^2$ has compact convex core.
Generalizing the correspondence between closed geodesics on $M^3$
and closed geodesics on $\Sigma^2$, we establish an orbit equivalence
between  recurrent spacelike geodesics on $M^3$ and
recurrent geodesics on $\Sigma^2$. In contrast, no timelike
geodesic recurs in either forward or backwards time.
\end{abstract}

\tableofcontents

\section*{Introduction}
A {\em Margulis spacetime\/} is a complete flat affine 
$3$-manifold $M^3$ with free nonabelian fundamental group 
$\Gamma$. 
It necessarily carries a unique parallel Lorentz metric.
Parallelism classes of timelike geodesics
form a noncompact complete hyperbolic surface $\Sigma^2$.
This complete hyperbolic surface is naturally associated to
the flat $3$-manifold $M^3$ and we regard $M^3$ as an
{\em affine deformation\/} of $\Sigma^2$.
This note relates the dynamics of the geodesic
flow of the flat affine manifold $M^3$ to the dynamics of the geodesic
flow on the hyperbolic surface $\Sigma^2$. 

We restrict to the case that $\Sigma^2$ has
a compact convex core 
(that is, $\Sigma^2$ has finite type and no cusps). 
Equivalently, the Fuchsian group $\Gamma_0$
corresponding to $\pi_1(\Sigma^2)$ is {\em convex cocompact. \/} 
In particular  $\Gamma_0$ is finitely generated and 
contains no parabolic elements. 
Under this assumption, every free homotopy
class of an essential closed curve in $\Sigma^2$ contains
a unique closed geodesic. 
Since $\Sigma^2$ and $M^3$ are homotopy-equivalent, 
free homotopy classes of essential closed curves in $M$
correspond to free homotopy classes of essential closed curves
in $\Sigma^2$. 
Every essential closed curve in $M^3$ is likewise homotopic
to a unique closed geodesic in $M^3$. 

In her thesis~\cite{Charette, CharetteGoldmanJones}, 
Charette studied the next case of dynamical behavior: 
geodesics spiralling around closed geodesics both in 
forward and backward time.
She proved bispiralling geodesics in $M^3$ exist,  
and correspond to bispiralling geodesics in $\Sigma^2$.  

This paper extends the above correspondence between geodesics on 
$\Sigma^2$ and $M^3$ to recurrent geodesics.

A geodesic (either in $\Sigma^2$ or in $M^3$) is {\em recurrent\/} 
if and only if it (together with its velocity vector) is recurrent in
{\em both\/} directions. These correspond to recurrent points for the
corresponding geodesic flows as in 
Katok-Hasselblatt~\cite{KatokHasselblatt}, \S 3.3.
(Our meaning of the term ``recurrent'' agrees with the
term ``nonwandering'' used by Eberlein~\cite{Eberlein}.)
Under our hypotheses on $\Sigma^2$,  a geodesic on $\Sigma^2$
is recurrent if and only if the corresponding orbit of the geodesic 
flow is precompact.

\begin{theorem}\label{thm:main}
Let $M^3$ be a Margulis spacetime whose associated
complete hyperbolic surface
 $\Sigma$ has compact convex core.
\begin{itemize}
\item
The recurrent part of the geodesic flow for  $\Sigma^2$ is 
{\em topologically orbit-equivalent\/} to the recurrent spacelike part of the geodesic flow of $M^3$.
\item
The set of recurrent spacelike geodesics in a Margulis spacetime is the closure of the set of periodic geodesics.
\item No timelike geodesic recurs.
\end{itemize}
\end{theorem}

A semiconjugacy between these flows was observed by 
D.\ Fried~\cite{Fried}. 

This note is the sequel to \cite{GoldmanLabourieMargulis},
which characterizes properness of affine deformations by
positivity of a marked Lorentzian length spectrum,
the  {\em generalized Margulis invariant\/}. 
A crucial step in the proof that properness implies positivity
is the construction of sections of the associated flat affine bundle,
called {\em neutralized sections.\/} A further modification of 
neutralized sections produces an orbit equivalence between
recurrent geodesics in $\Sigma$ and recurrent geodesics in $M$.

It follows that the  set of recurrent spacelike orbits of the geodesic
flow is a Smale hyperbolic set in $\TM$.

Null geodesics not parallel to a point in the limit set $\Lambda$ 
of $\Gamma_0$ do not recur. 
We do not discuss the recurrence of  null geodesics parallel 
to a point of $\Lambda$  in this paper,
but suspect that  such null geodesics do not recur either.
Therefore we conjecture that the only recurrent orbits
of the orbit flow are spacelike.

We thank  Mike Boyle, Virginie Charette, Suhyoung Choi,
Todd Drumm, David Fried, and
Gregory Margulis  for helpful conversations. We are grateful to
Domingo Ruiz for pointing out several corrections.

\section{Geodesics on affine manifolds}

An {\em affinely flat manifold\/} is a smooth manifold with a
distinguished
atlas of local coordinate systems whose charts map to 
an affine space $\E$ such that the coordinate changes
are restrictions of affine automorphisms of $\E$.
Denote the group of affine automorphisms of $\E$ by
$\Aff(\E)$.
This structure is equivalent to a flat torsionfree affine connection.  
The affine coordinate atlas globalizes to a {\em developing map\/}
\begin{equation*}
\tilde{M} \xrightarrow{\dev} \E
\end{equation*}
where $\tilde{M}\to M$ denotes a universal covering space of $M$.
The coordinate changes globalize to 
an affine holonomy homomorphism
\begin{equation*}
\pi_1(M) \xrightarrow{\rho} \Aff(\E) 
\end{equation*}
where $\pi_1(M)$ denotes the group of deck transformations
of $\tilde{M}\to M$.  The developing map is equivariant
respecting $\rho$.

Denote the vector space of translations $\E\to\E$ by $\V$. 
The action of $\V$ by translations on $\E$ defines a trivialization
of the tangent bundle $\TM \cong M \times \V$.
In these local coordinate charts, a geodesic is a path
\begin{equation*}
p \longmapsto  p + t \vv
\end{equation*}
where $p\in\E$ and $\vv\in\V$ is a vector.
In terms of the trivialization the geodesic flow is:
\begin{align*} 
\E\times\V &\xrightarrow{\tilde\psi_t}  \E\times\V \\
(p,\vv) &\longmapsto  (p + t \vv, \vv)
\end{align*}
for $t\in\R$. Clearly this $\R$-action commutes with $\Aff(\E)$.

Geodesic completeness implies that $\dev$ is a diffeomorphism.
Thus the universal covering $\tilde M$ 
is affinely isomorphic to affine space $\E$ and 
$M\cong \E/\Gamma$, 
where $\Gamma := \rho\big(\pi_1(M)\big)$ 
is a discrete group of affine transformations 
acting properly and freely on $\E$. 

\section{Flat Lorentz $3$-manifolds}

Let $\Aff(\E) \xrightarrow{\L} \mathsf{GL}(\V)$
denote the homomorphism given by {\em linear part,\/}
that is, $\L(\gamma) = A$ where 
\begin{equation*}
p \xrightarrow{\gamma}A (p)  + b.
\end{equation*}
The differential of $\gamma$  at {\em any\/} point $p$ 
identifies with its linear part $\L(\gamma)$ via the identification
$\TM \cong M \times \V$.

Any $\L(\Gamma)$-invariant nondegenerate inner product 
$\langle,\rangle$
on $\V$ defines a $\Gamma$-invariant flat pseudo-Riemannian structure on $\E$ which descends to $M = \E/\Gamma$. 
In particular affine manifolds with  
$\L(\Gamma)\subset\mathsf{O}(n-1,1)$ are precisely 
the {\em flat Lorentzian manifolds,\/} and the underlying affine structures
their Levi-Civita connections.

For this reason we henceforth fix the invariant Lorentzian inner
product on $\V$, and hence the (parallel) flat Lorentzian structure
on $\E$. The group $\Isom(\E)$ of Lorentzian isometries is the
semidirect product of the group $\V$ of translations of $\E$ with
the orthogonal group $\SOno$ of linear isometries. 
Linear part $\Isom(\E) \xrightarrow{\L}\SOno$ defines the projection
homomorphism for the semidirect product.
For $l\in\R$, 
define 
\begin{equation*}
\Ss_l \;:=\; \{ \vv\in\V \mid \langle \vv,\vv\rangle\ =\ l \}.
\end{equation*}
When $l > 0$, $\Ss_l$ is a Riemannian submanifold of constant
curvature $-l^{-2}$ and when $l<0$, 
it is a Lorentzian submanifold of constant curvature $l^{-2}$.
In particular $\Ss_{-1}$ is a disjoint union of two isometrically
embedded copies of hyperbolic $n-1$-space $\mathsf{H}^{n-1}$ 
and $\Ss_1$ is {\em de Sitter space,\/} 
a model space of Lorentzian curvature $+1$.

The subset $\T_l(M)$ consisting of tangent vectors $v$ such that $\langle v,v\rangle = l$ is invariant under the geodesic flow.
Indeed, using parallel translation, these bundles trivialize over
the universal covering $\E$:
\begin{equation*}
\T_l(\E) \xrightarrow{\cong}  \E \times \Ss_l
\end{equation*}

Abels-Margulis-Soifer~\cite{AbelsMargulisSoifer1,
AbelsMargulisSoifer2} proved that 
if a discrete group of Lorentz isometries acts properly on 
Minkowski space $\E$, and $\L(\Gamma)$ is Zariski
dense in $\SOno$, then $n=3$. Consequently every
complete flat Lorentz manifold is a flat Euclidean affine fibration over
a complete flat Lorentz $3$-manifold. 
Thus we henceforth restrict to $n=3$.

Let $M^3$ be a complete affinely flat $3$-manifold.
By Fried-Goldman~\cite{FriedGoldman}, either $\Gamma$ is solvable
or $\L\circ h$  
embeds $\Gamma$ as a discrete subgroup in (a conjugate of)
the  orthogonal group
\begin{equation*}
\SOto \subset \mathsf{GL}(3,\R).
\end{equation*}
The cases when $\Gamma$ is solvable are easily classified
(see \cite{FriedGoldman}) and we assume we are in the latter case.
In that case, $M^3$ is a complete flat Lorentz $3$-manifold.

In the early 1980's Margulis, answering a question of Milnor~\cite{Milnor}), 
constructed the first examples~\cite{Margulis1,Margulis2}, 
which are now called {\em Margulis spacetimes.}
Explicit geometric constructions of these manifolds have been
given by Drumm~\cite{Drumm1,Drumm2} and his 
coauthors~\cite{Charette,Charette1,CharetteDrummGoldman,CharetteGoldman,
DrummGoldman}.

Since the hyperbolic plane $\Ht$ is the symmetric space of $\SOto$,
$\Gamma$ acts properly and discretely on $\Ht$.
Since $M^3$ is aspherical, its fundamental group $\pi_1(M^3)\cong\Gamma$ is torsionfree, so $\Gamma$ acts freely as well.
Therefore the quotient $\Ht/\L(\Gamma)$ is a complete hyperbolic
surface $\Sigma^2$. Furthermore, 
by Mess~\cite{Mess},  $\Sigma$ is noncompact.
(See Goldman-Margulis~\cite{GoldmanMargulis} and Labourie~\cite{Labourie} for alternate proofs.)
Furthermore 
every noncompact complete hyperbolic surface occurs for a Margulis
spacetime (Drumm~\cite{Drumm1})

The points of $\Sigma^2$ correspond to parallelism classes of 
(unoriented) timelike geodesics on $M^3$ as follows. 
It suffices to identify $\Ht$ with the parallelism classes of (unoriented) timelike geodesics in $\E$, equivariantly respecting 
$\Isom(\E) \xrightarrow{\L}\SOto$.
The velocity of a unit-speed timelike geodesic is in $\E$ 
is  a $\tilde\psi$-orbit in 
\begin{equation*}
\T_{-1}\E \;\cong\; \big(\E \times \Ss_{-1}\big).
\end{equation*}
The two components of $\Ss_{-1}$  correspond to to future-pointing
timelike geodesics and past-pointing timelike geodesics
respectively. Points in $\Ss_{-1}$ correspond to points in 
$\Ht$ (the projectivization of $\Ss_{-1}$) together with an orientation
of $\Ht$.
The geodesic flow $\tilde\psi$ gives $\T_{-1}\E$ the structure of a principal $\R$-bundle over the quotient. The quotient identifies with
an affine bundle over $\Ss_{-1} \cong \Ht \times \{\pm 1\}$ 
whose associated vector bundle is the tangent bundle, as follows.
The space of lines parallel to
a fixed timelike vector $\vv$ with the quotient affine space, whose
underlying vector space is $\V/(\vv)\ \cong\ (\vv)^\perp$.
The tangent space to $\Ss_{-1}$ at $\vv$ is $\vv^\perp$ proving the claim.

Passing to the quotient by $\Gamma$:
\begin{equation*}
\T_{-1}M \;\cong\; \big(\E \times \Ht\big)/\Gamma.
\end{equation*}

Since $\Gamma\xrightarrow{\L}\SOto$ is a discrete 
embedding~\cite{FriedGoldman},
$\SOto$ acting properly on $\Ht$ implies that $\Gamma$ acts
properly on $\Ht$.   
Cartesian projection $\E\times\Ht \to \Ht$ induces
a projection 
\begin{equation*}
\T_{-1}M \;\longrightarrow \Ht/\L(\Gamma) \;=\; \Sigma,
\end{equation*}
invariant under the restriction of the geodesic flow $\psi$ 
to $\T_{-1}M$, 
which defines an $\E$-bundle over $\Sigma$.
Its fiber over the orbit $\Gamma\vv$ of
a fixed future-pointing unit-timelike vector $\vv$ 
is the union of  geodesics in $M = \E/\Gamma$ parallel to 
$\Gamma\vv$. In particular  properness of the 
$\L(\Gamma)$-action on $\Ht$ implies nonrecurrence of timelike geodesics, the last statement in Theorem~\ref{thm:main}.

More generally, any $\L(\Gamma)$-invariant subset 
$\Omega\subset\V$ defines a subset
$\T_\Omega(M) \subset \TM$ invariant under the geodesic flow.
If $\Omega$ is an open set upon which $\L(\Gamma)$ acts properly, then the geodesic flow defines a proper $\R$-action
on $\T_\Omega(M)$.
In particular every geodesic whose velocity lies in $\Omega$
is properly immersed and is neither positively nor negatively recurrent.

An important example is the following. 
The lines in $\Ss_0$ form the {\em the ideal boundary,\/}
(the circle-at-infinity) $\partial\Ht$, of $\Ht$.
The {\em limit set\/} of $\L(\Gamma)$ consists of endpoints
of recurrent geodesic rays in $\Sigma$.
Furthermore $\Lambda_{\L(\Gamma)}$ is the unique closed
$\L(\Gamma)$-invariant closed subset of $\partial\Ht$.
In particular the set of fixed points of elements of $\L(\Gamma)$
is dense in $\Lambda_{\L(\Gamma)}$. 
Moreover $\L(\Gamma)$ acts properly on the complement 
$$
\Omega := \Ss_0\setminus \Lambda_{\L(\Gamma)}.
$$ 
Applying the above discussion, no geodesic tangent
to  $\T_\Omega(M)$ recurs. That is, a lightlike recurrent
geodesic ray must be parallel to 
$\Lambda_{\L(\Gamma)}$.


\section{From  geodesics in $\Sigma^2$ to geodesics in $M^3$}  

While timelike directions correspond to points of $\Sigma^2$,
spacelike directions correspond to geodesics in $\Ht$. 
The recurrent geodesics in $\Sigma$ intimately relate to the
recurrent spacelike geodesics on $M^3$. 

Denote the set of oriented spacelike geodesics in $\E$ by $\mathscr{S}$.
It identifies with the orbit space of the  geodesic flow $\tilde\psi$
on $\T_{+1}\E \cong \E \times \dS$. The natural map 
$\mathscr{S}\xrightarrow{\Upsilon}\dS$ which associates to a 
spacelike vector its direction is equivariant respecting
$\Isom(\E)\xrightarrow{\L}\SOto$.

The identity component of $\SOto$ acts simply transitively on 
the unit tangent bundle $\U\Ht$, and therefore we identify
$\SOto^0$ with $\U\Ht$ by choosing a basepoint $u_0$ in $\U\Ht$.
Unit-spacelike vectors in $\Ss_{+1}$ correspond to oriented geodesics in $\Ht$.
Explicitly, if $v\in\dS$, then there is a one-parameter subgroup $a(t)\in\SOto$, having $v$ as a fixed vector, and such that 
$$
\det(v,v^-,v^+)\ >\ 0,
$$
where $v^+$ is an expanding eigenvector of $a(t)$ (for $t>0$) and
$v^-$ is the contracting eigenvector.  Choose a basepoint 
$v_0\in\Ss_{+1}$ corresponding to the orbit of $u_0$ under the geodesic flow on $\US$.
Geodesics in $\Ht$ relate to spacelike directions by an equivariant mapping 
\begin{align*}
\U \Ht &\;
\longrightarrow 
\Ss_{+1} \\
g(u_0) &\longmapsto g(v_0)
\end{align*}
The unit tangent bundle $\US$ of $\Sigma$ identifies with the
quotient 
\begin{equation*}
\L(\Gamma)\backslash\U\Ht \;\cong\; \L(\Gamma)\backslash\SOto^0,
\end{equation*}
where the geodesic flow $\psi$ corresponds the right-action of $a(-t)$
(see, for example, \cite{GoldmanLabourieMargulis},\S 1.2).

Observe that a geodesic in $\Sigma^2$ is recurrent if and only if the endpoints of any of its lifts to $\tS\approx\Ht$  lie in the limit set
$\Lambda_{\L(\Gamma)}$ of $\dG$.   
If the convex core of $\Sigma^2$ is compact, 
then the union $\Ur$  of recurrent $\phi$-orbits is compact.

\begin{lemma}\label{MainProp} 
There exists an orbit-preserving map 
\begin{equation*}
\Ur \xrightarrow{\bold \hN}\UpM
\end{equation*}
mapping $\phi$-orbits injectively to recurrent $\psi$-orbits.
\end{lemma}

\begin{proof} 
The associated flat affine bundle $\EE_\Gamma$ over $\US$
associated to the affine deformation $\Gamma$ 
is defined as follows.
The affine representation of $\Gamma$ defines a diagonal action
of $\Gamma$ on $\widetilde{\US} \times \E$.
Its total space is the quotient of the product
$\widetilde{\US} \times \E$
by the diagonal action of $\pi_1(\US)$: 
\begin{equation*}
\pi_1(\US)\longrightarrow\pi_1(\Sigma)\longrightarrow\Isom(\E).
\end{equation*}
Similarly the flat vector bundle $\V_\Gamma$
over $\US$ is the quotient of $\widetilde{\US} \times \V$ 
by the diagonal action:
\begin{equation*}
\pi_1(\US)\longrightarrow\pi_1(\Sigma)\longrightarrow\Isom(\E)\xrightarrow{\L} \SOto.
\end{equation*}
According to \cite{GoldmanLabourieMargulis}, 
the {\em neutral section\/} of $\V_\Gamma$ is a 
$\SOto$-invariant section
which is parallel with respect to the geodesic flow 
on $\US$. and arises from the graph of the 
$\SOto$-equivariant map 
\begin{equation*}
\U\tS \cong \U\Ht \longrightarrow \V
\end{equation*}
with image $\dS$, the space of unit-spacelike vectors in $\V$. 

Here is the main construction of \cite{GoldmanLabourieMargulis}.
To every section $\sigma$ of $\EE_\Gamma$ continuously 
differentiable along $\phi$, associate the function 
$$
F_\sigma \; :=\;  \langle \nabla_\phi\sigma, \nu \rangle
$$
on $\US$. (Here the covariant derivative of a section of $\EE_\Gamma$ along a vector field $\phi$ in the base is a section of the associated vector bundle
$\V_\Gamma$.)  Different choices of section $\sigma$ yield cohomologous functions $F_{\sigma}$.
(Recall that two functions $f_1,f_2$ are {\em cohomologous,\/}
written $f_1 \sim f_2$,  if 
\begin{equation*}
f_1 - f_2 \;=\; \phi\ g
\end{equation*}
for a function $g$ which is differentiable with respect to the vector field 
$\phi$ (\cite{KatokHasselblatt},\S 2.2).

Restrict the affine bundle $\EE_\Gamma$  to $\Ur$. 
Lemma~8.4 of \cite{GoldmanLabourieMargulis} 
guarantees the existence of a {\em neutralized section\/}, 
that is, a section $N$ of $(\EE_\Gamma)|_{\Ur}$ satisfying
\begin{equation*}
\nabla_\phi N=f\nu,
\end{equation*}
for some function $f$.

Although the following lemma is well known,
we could not find a proof in the literature.
For completeness, we supply a proof in the appendix.

\begin{lemma}\label{lemma:positive}
Let $X$ be a compact space equipped with a flow $\phi$. 
Let $f\in C(X)$, such that for  all $\phi$-invariant measures $\mu$
on $X$, 
\begin{equation*}
\int f\ \d\mu \;>\;0.
\end{equation*}
Then  $f$ is  cohomologous to a  positive function.
\end{lemma}

Since $\Gamma$ acts properly, Proposition~8.1  of 
\cite{GoldmanLabourieMargulis}
implies that $\int F_\sigma d\mu \neq  0$ for all $\phi$-invariant probability measures $\mu$ on $\Ur$.
Since the set of invariant measures is connected, 
$\int F_\sigma d\mu$ is either positive for all 
$\phi$-invariant probability measures $\mu$ on $\Ur$ or negative for all
$\phi$-invariant probability measures $\mu$ on $\Ur$.
Conjugating by $-I$ if necessary we may assume that
$\int F_\sigma d\mu > 0$.
Lemma~\ref{lemma:positive} implies
$F_\sigma + \phi g >0$ for some function $g$. Write 
\begin{equation*}
\hN\;=\;\ N +\ g \nu.
\end{equation*}
$\hN$ remains neutralized, and  $\nabla_\phi\hN$ vanishes nowhere.

Let $\tUr$ be the preimage of $\Ur$ in ${\mathsf U}\Ht$.
Then $\hN$ determines a $\Gamma$-equivariant map 
\begin{equation*}
\tUr\xrightarrow{\thN}\E.
\end{equation*}
Each $\tilde\phi$-orbit injectively maps to a spacelike geodesic.
The map 
\begin{align*}
\Ur &\xrightarrow{\bold{\hN}} \big(\E\times\dS\big)/\Gamma \\
x &\longmapsto
\big[(\hN(x), \nu(x)) \big].
\end{align*}
is the desired orbit equivalence $\Ur\longrightarrow\UpM$.
\end{proof}
\begin{lemma}\label{lemma:parallel}
Any spacelike recurrent geodesic parallel to a geodesic 
$\gamma$ in the image of $\bold \hN$ coincides with $\gamma$.
\end{lemma}
\begin{proof}
Let 
$t\stackrel{g}\longmapsto \phi_t(v)$ be an orbit in $\Ur$.
A geodesic $\xi$ parallel to $\bold \hN(g)$ determines a parallel section $u$ of $\V$ along $g$. Since $g$ recurs, the
resulting parallel  section is a bounded invariant parallel section along the closure of $g$. By the Anosov property, such a section is along $\nu$,
and therefore, up to reparametrization,  $\gamma=\bold \hN(g)$.
\end{proof}

\begin{proposition} 
$\bold \hN$ is injective and its image is the set of recurrent spacelike geodesics.
\end{proposition}

\begin{proof}
An orbit of the geodesic flow $\phi$  recurs
if and only if the corresponding $\Gamma$-orbit in the 
space $\mathscr{S}$ of spacelike geodesics in $\E$ recurs.
Similarly a $\phi$-orbit in $\UpM$ recurs if and only if the corresponding
$\L(\Gamma)$-orbit in $\dS$ recurs.
The map 
$\mathscr{S} \xrightarrow{\Upsilon} \dS$
recording the direction of a spacelike geodesic
is $\L$-equivariant.
If the $\Gamma$-orbit  of $g\in\mathscr{S}$ corresponds to
a recurrent spacelike geodesic in $M$,
then the $\L(\Gamma)$-orbit of $\Upsilon(g)$ corresponds to
a recurrent $\phi$-orbit in $\US$.

$\bold \hN$ is injective along orbits of the geodesic flow. 
Thus  it suffices  to prove that the restriction
of $\Upsilon$ to the subset
of $\Gamma$-recurrent geodesics in $\mathscr{S}$  
is injective. Since the fibers of $\Upsilon$ are parallelism classes
of spacelike geodesics, Lemma~\ref{lemma:parallel} implies
injectivity of $\bold\hN$.

Finally let $g$ be a $\psi$-recurrent point in 
$\UpM$, corresponding to a spacelike  recurrent geodesic $\gamma$
in $M$. It corresponds to a recurrent $\Gamma$-orbit $\Gamma g$ in 
$\mathscr{S}$. Then $\Upsilon(\Gamma g)$ is a recurrent $\L(\Gamma)$-orbit in $\dS$, and corresponds to a recurrent $\phi$-orbit in $\US$.
The image of this $\phi$-orbit under $\bold\hN$ is a  spacelike recurrent
geodesic in $\UpM$ parallel to $\gamma$. Now apply 
Lemma~\ref{lemma:parallel} again to conclude that $g$ lies in the image of $\bold\hN$.
\end{proof}
The proof of Theorem~\ref{thm:main} is complete.

\section{Appendix: Cohomology and positive functions}


Let $X$ be a smooth manifold equipped with a smooth flow $\phi$.
A function $g\in C(X)$ is {\em continuously differentiable along $\phi$\/}
if for each $x\in X$, the function
$$
t \mapsto  g\big(\phi_t(x)\big)
$$
is a continuously differentiable map $\R\longrightarrow X$. 
Denote the subspace of $C(X)$ consisting of functions continuously
differentiable along $\phi$ by $C_\phi(X)$. For $g\in C_\phi(X)$,
denote its directional derivative by:
$$
\phi(g) \;:=\;  \frac{d}{dt}\bigg|_{t=0} g\circ\phi_{t}.
$$
The proof of Lemma~\ref{lemma:positive} will be based on two lemmas.
\begin{lemma}\label{lem:averaging_cohomology}
Let $f\in C_\phi(X)$. 
For any  $T>0 $, define
\begin{equation*}
f_T(x)\;:=\; \frac{1}{T}\int_0^T f\big(\phi_s(x)\big)\ {\d}s.
\end{equation*}
Then $f \sim f_T$.
\end{lemma}
\begin{proof}
We must show that there exists a function  $g\in C_\phi(X)$ such that:
\begin{equation*}
f_T - f  \;=\;  \phi g .  
\end{equation*}
By the fundamental theorem of calculus,
\begin{equation*}
f\circ \phi_{t}  \;=\;  f + \int_0^t (\phi f\circ\phi_{s}) \ {\d}s.
\end{equation*}
Writing
\begin{equation*}
g = \frac1{T} \int_0^T  \int_0^t (f \circ \phi_{s}) \ {\d}s\ {\d}t, 
\end{equation*}
then
\begin{align*}
f_T - f & =   \frac1{T} \int_0^T (f\circ \phi_{t}   - f)    \;{\d}t \\
& = \frac1{T} \int_0^T  \int_0^t \phi(f \circ \phi_{s})
 \; {\d}s\; {\d}t \\
& = \phi g .
\end{align*}
as desired.\end{proof}

\begin{lemma}\label{lem:averaging_positive}
Assume that for all $\phi$-invariant measures $\mu$, 
$$
\int f{\d}\mu >0.
$$
Then 
$f_T>0$ for some $T>0$.
\end{lemma}
\begin{proof} 
Otherwise  sequences $\seq{T}$ of positive real numbers and
sequences 
$\seq{x}$ of points in $M$ exist such that 
$$
f_{T_n}(x_n)\leq 0.
$$ 
Using the flow $\phi_t$, push forward 
the (normalized) Lebesgue measure
\begin{equation*}
\frac1{T_m}\mu_{[0,T_m]}
\end{equation*}  
on the interval $[0,T_m]$ to $X$,
to obtain a sequence of probability measures $\mu_n$
on $X$ such that
\begin{equation*}
\int  f\ {\d}\mu_n \;\leq\;  0.
\end{equation*}
As in \cite{GoldmanLabourieMargulis}, \S 7, 
a subsequence weakly converges to an 
$\phi$-invariant measure $\mu$ for which
$$
\int f\ {\d}\mu\leq 0,
$$
contradicting our hypotheses.
\end{proof}

\begin{proof}[Proof of Lemma~\ref{lemma:positive}]
By Lemma~\ref{lem:averaging_cohomology}, $f\sim f_T$ for
any $T>0$, and Lemma~\ref{lem:averaging_positive} implies
that  $f_T>0$ for some $T$.
\end{proof}

 \end{document}